\documentclass[10pt]{amsart}%{jarticle}
\usepackage{amsmath,amsthm, epsfig, amscd}
\usepackage{psfrag}
\usepackage[psamsfonts]{amssymb}
\usepackage[all]{xy}
\usepackage{mathrsfs}

\numberwithin{equation}{section}

\newtheorem{Thm}{Theorem}[section]
\newtheorem{Prop}[Thm]{Proposition}
\newtheorem{Lem}[Thm]{Lemma}

\newtheorem{Cor}[Thm]{Corollary}

\theoremstyle{remark}
\newtheorem{Rem}[Thm]{Remark}

\theoremstyle{definition}

\newtheorem{Assum}[Thm]{Assumption}

\newcommand{\mysection}[2]{%
\vspace{2mm}\section{\bf #1}\label{#2}
}

\def\Z{{\mathbb Z}}
\def\R{{\mathbb R}}
\def\Q{{\mathbb Q}}
\def\C{{\mathbb C}}
\def\calA{\mathscr{A}}

\def\calD{\mathscr{D}}

\def\deg{\mathrm{deg}}

\def\Conf{\mathrm{Conf}}
\def\bConf{\overline{\Conf}}
\def\bC{\overline{C}}

\def\Diff{\mathrm{Diff}}

\newcommand{\mapright}[1]{
	\smash{\mathop{
		\hbox to 1cm{\rightarrowfill}}\limits^{#1}}}

\newcommand{\mapleft}[1]{
	\smash{\mathop{
		\hbox to 1cm{\leftarrowfill}}\limits^{#1}}}

\def\Z{\mathbb{Z}}
\def\Q{\mathbb{Q}}
\def\C{\mathbb{C}}
\def\R{\mathbb{R}}
\def\F{\mathbb{F}}
\def\calA{\mathscr{A}}
\def\calD{\mathscr{D}}

\def\End{\mathrm{End}}

\def\Tr{\mathrm{Tr}}

\def\ve{\varepsilon}

\def\even{\mathrm{even}}
\def\odd{\mathrm{odd}}
\def\Sym{\mathrm{Sym}}

\def\tbigwedge{\textstyle\bigwedge}
\def\Map{\mathrm{Map}}

\begin{document}

\title[Unstable pseudo-isotopies of spherical 3-manifolds]{Unstable pseudo-isotopies of spherical 3-manifolds}
\author{Yuji Ohta}
\author{Tadayuki Watanabe}

\date{\today}
%\subjclass[2000]{57M27, 57R57, 58D29, 58E05}

{\noindent\footnotesize {\rm Preprint} (2023)}\par\vspace{15mm}
\maketitle
\vspace{-6mm}
\setcounter{tocdepth}{2}
\begin{abstract}
In our previous works, we constructed diffeomorphisms of compact 4-manifolds $X$ by surgeries on theta-graphs embedded in $X$. In this paper, we consider the case $X=M\times I$, where $M$ is a spherical 3-manifold. For some of such $X$, we compute lower bounds of the ranks of the abelian groups $\pi_0\Diff(X,\partial)$. We study the behavior of the elements constructed by theta-graph surgery under the suspension functor in stable pseudo-isotopy theory, and their triviality in the space of block diffeomorphisms.
\end{abstract}
\par\vspace{3mm}
%\tableofcontents

%%%%%%%%%%%%%%%%%%%%%%%%%%%%%%
%%%%%%%%%%%%%%%%%%%%%%%%%%%%%%
%%%%%%%%%%%%%%%%%%%%%%%%%%%%%%
\def\baselinestretch{1.06}\small\normalsize
\mysection{Introduction}{s:intro}

For a finite group $\pi$, let $(\pi\times\pi)\rtimes \Z_2$ be the semidirect prodct of $\pi\times\pi$ and $\Z_2$ with respect to the homomorphism $\psi\colon \Z_2=\{1,\tau\}\to \mathrm{Aut}(\pi\times \pi)$; $\tau\mapsto ((x,y)\mapsto (y,x))$. For a left $(\pi\times\pi)\rtimes \Z_2$-module $W$, let 
\[ \begin{split}
  &\calA_\Theta^\even(W):=H^0((\pi\times\pi)\rtimes \Z_2;\tbigwedge^3 W)=
\bigl(\tbigwedge^3 W\bigr)^{(\pi\times\pi)\rtimes \Z_2},\\
  &\calA_\Theta^\odd(W):=H^0((\pi\times\pi)\rtimes \Z_2;\Sym^3 W)=\bigl(\Sym^3\,W\bigr)^{(\pi\times\pi)\rtimes \Z_2},
\end{split} \]
where we take the invariant parts. One can take $W=\C[\pi]$, which is both a left $\pi\times\pi$-module by $(g,h)\cdot x\mapsto gxh^{-1}$, and a $\Z_2$-module by the involution $x\otimes y\otimes z\mapsto x^{-1}\otimes y^{-1}\otimes z^{-1}$ ($x,y,z\in\pi$). Also, one can take $W=\mathrm{Ker}\,\ve$ for the augmentation map $\ve:\C[\pi]\to \C$, which is a $(\pi\times\pi)\rtimes\Z_2$-submodule of $\C[\pi]$. The definitions of the $\calA_\Theta$-spaces are motivated by the space of $\pi$-decorated 2-loop graphs in \cite{GL}. Let $\Diff_0(X,\partial)$ denote the subgroup of $\Diff(X,\partial)$ of diffeomorphisms homotopic to the identity.
In this paper, we consider the 4-manifold $X=M\times I$, where $M$ is a spherical (or elliptic) 3-manifold, i.e. $M=S^3/\Gamma$ for a finite subgroup $\Gamma\subset SO(4)$ acting freely on $S^3$ (e.g., \cite[\S{1.2.1}]{Sa}). 
The following theorem is essentially given in \cite[Proposition~7.1 and Remark~7.2]{Wa20}.

\begin{Thm}\label{thm:1}
Let $M$ be a spherical 3-manifold, $\pi=\pi_1M$, and let $X=M\times I$. 
Let $W$ be a $\C[(\pi\times\pi)\rtimes\Z_2]$-module satisfying Assumption~\ref{assum:2-acyclic}. Then a homomorphism
\[ Z_\Theta^\even\colon \Gamma_0(X)\otimes\C\to \calA_\Theta^\even(W) \]
from a certain subgroup $\Gamma_0(X)$ of $\pi_1 B\Diff_0(X,\partial)$ is defined, and if $W=\mathrm{Ker}\,\ve$, it is surjective. Hence we have
\[ \dim \pi_0\Diff_0(X,\partial)\otimes\Q\geq \dim\calA_\Theta^\even(\mathrm{Ker}\,\ve). \]
\end{Thm}

Note that for $X=M\times I$, the group $\Diff(X,\partial)$ is homotopy commutative (\cite[2.6.1~Lemma]{ABK}) and hence $\pi_0\Diff_0(X,\partial)$ is abelian. Theorem~\ref{thm:1} can be obtained by modifying a few points (\S\ref{s:homol} and \S\ref{s:Z} below) in \cite[Proposition~7.1 and Remark~7.2]{Wa20}, where we defined an invariant of diffeomorphisms of $X$ with certain structures by a method similar to \cite{Mar,Les1,Les2} (``equivariant triple intersection'' in a configuration space). The aim of this paper is to give some examples of computations of the lower bound in Theorem~\ref{thm:1}. 

The lower bound $\dim\calA_\Theta^\even(\mathrm{Ker}\,\ve)$ of Theorem~\ref{thm:1} can be computed by simple calculations. We give a few examples: the Poincar\'{e} homology sphere $\Sigma(2,3,5)$ and the lens spaces $L(n,q)$. It is known that $\pi_1\Sigma(2,3,5)$ is isomorphic to $\mathrm{SL}_2(\F_5)$, which is also known as the binary icosahedral group, and $\pi_1L(n,q)\cong \Z_n$ (e.g., \cite[\S{1.1.2}, \S{1.2.1}]{Sa}). %The following computation using the characters for the binary icosahedral group was done with the help of Yuji Ohta.

\begin{Prop}[Poincar\'{e} homology sphere, Proposition~\ref{prop:ex1}]
When $\pi=\mathrm{SL}_2(\F_5)$, we have the following.
\begin{enumerate}
\item $\dim \calA_\Theta^\even(\mathrm{Ker}\,\ve)=\dim \calA_\Theta^\even(\C[\pi])=27$.
\item $\dim \calA_\Theta^\odd(\mathrm{Ker}\,\ve)=56$, $\dim \calA_\Theta^\odd(\C[\pi])=65$.
\end{enumerate}
\end{Prop}

The case of lens spaces can be computed by a more elementary method.
\begin{Prop}[Lens spaces, Proposition~\ref{prop:ex2}]
Let $\pi=\Z_n$ ($n\geq 1$), and for an integer $m\geq 0$, let $p_3(m)$ be the number of partitions of $m$ into at most three parts, namely, the number of integer solutions of the equation $x+y+z=m$ ($0\leq x\leq y\leq z$). We set $p_3(m)=0$ for $m<0$. Then we have the following.
\begin{enumerate}
\item $\dim\calA_\Theta^{\odd}(\C[\pi])=p_3(n)$.
\item $\dim\calA_\Theta^{\even}(\C[\pi])=p_3(n-6)$.
\item $\dim\calA_\Theta^{\odd}(\mathrm{Ker}\,\ve)=p_3(n-3)$.
\item $\dim\calA_\Theta^{\even}(\mathrm{Ker}\,\ve)=\dim\calA_\Theta^{\even}(\C[\pi])=p_3(n-6)$.
\end{enumerate}
\end{Prop}
The following is a table of the values of the dimensions of $\calA_\Theta^{\even/\odd}(\C[\pi])$ and $\calA_\Theta^{\even/\odd}(\mathrm{Ker}\,\ve)$ for $n\leq 15$. It is known that $p_3(m)$ ($m\geq 0$) is the nearest integer to $\frac{(m+3)^2}{12}$ (e.g. \cite[Ch. 6]{AK}).
\par\medskip
\begin{center}
\begin{tabular}{|c|ccccccccccccccc|}\hline
$n$ & 1 & 2 & 3 & 4 & 5 & 6 & 7 & 8 & 9 & 10 & 11 & 12 & 13 & 14 & 15\\ \hline
$\dim\calA_\Theta^{\odd}(\C[\pi])$ &
  1 & 2 & 3 & 4 & 5 & 7 & 8 & 10 & 12 & 14 & 16 & 19 & 21 & 24 & 27\\
$\dim\calA_\Theta^{\even}(\C[\pi])$ & 
  0 & 0 & 0 & 0 & 0 & 1 & 1 & 2 & 3 & 4 & 5 & 7 & 8 & 10 & 12\\ 
$\dim\calA_\Theta^{\odd}(\mathrm{Ker}\,\ve)$ &
  0 & 0 & 1 & 1 & 2 & 3 & 4 & 5 & 7 & 8 & 10 & 12 & 14 & 16 & 19\\ 
$\dim\calA_\Theta^{\even}(\mathrm{Ker}\,\ve)$ & 
  0 & 0 & 0 & 0 & 0 & 1 & 1 & 2 & 3 & 4 & 5 & 7 & 8 & 10 & 12\\ \hline
\end{tabular}
\end{center}
\par\medskip

Although the results for $\calA_\Theta^\odd$ are not used in this paper, we think these would be useful to study finite type invariants of 3-manifolds in \cite{GL}.

Let $C(M)=\Diff(M\times I, M\times \{0\}\cup \partial M\times I)$, the group of diffeomorphisms of $M\times I$ pointwise fixing $M\times \{0\}\cup \partial M\times I$, equiped with the $C^\infty$-topology. An element of $C(M)$ is called a {\it pseudo-isotopy} or a {\it concordance} of $M$. Pseudo-isotopy theory (e.g., \cite{Ce,HW,Ig1}) studies the topology of $C(M)$, which is related to the diffeomorphism groups via the following fiber sequence
\begin{equation}\label{eq:C(M)}
 \Diff(M\times I,\partial) \to C(M) \stackrel{r}{\to} \Diff(M), 
\end{equation}
where $r$ is the restriction to $M\times\{1\}$. 
For most 3-manifolds $M$, the left map in this sequence induces a map between $\pi_0$ which is close to an isomorphism, in the sense that $\pi_i\Diff(M)$ is small in many cases (generalized Smale conjecture, e.g., \cite{Hat,Ga,HKMR,BK}). 

\begin{Cor}
We have the following inequalities.
\begin{enumerate}
\item $\dim \pi_0C(\Sigma(2,3,5))\otimes\Q\geq 27$.
\item $\dim \pi_0C(\Sigma(2,3,5)\times I)\otimes \Q\geq 27$.
\item $\dim \pi_0C(L(n,q)\times I)\otimes \Q\geq p_3(n-6)$.
\end{enumerate}
\end{Cor}
\begin{proof}
For (1), since $\Diff(\Sigma(2,3,5))\simeq \mathrm{Isom}(\Sigma(2,3,5))=SO(3)$ by \cite{BK}, and $\pi_0SO(3)=0$, the natural map $\pi_0\Diff(\Sigma(2,3,5)\times I,\partial)\to \pi_0 C(\Sigma(2,3,5))$ from (\ref{eq:C(M)}) is surjective, and the image is abelian, too. Moreover, its kernel is the image from $\pi_1SO(3)=\Z_2$. Hence we have an isomorphism
$\pi_0\Diff(\Sigma(2,3,5)\times I,\partial)\otimes\Q\cong \pi_0 C(\Sigma(2,3,5))\otimes\Q$. Then the result follows by Theorem~\ref{thm:1} and Proposition~\ref{prop:ex1}.

For (2) and (3), we use the fact that theta-graph surgeries that are detected by $Z_\Theta^\even$ lifts to $\pi_1BC(M\times I)$ (\cite[Theorem~1.3]{BW}). Then the results follow by Theorem~\ref{thm:1} and Propositions~\ref{prop:ex1}, \ref{prop:ex2}.
\end{proof}

We compare our nontrivial subgroup in $\pi_0 C(M\times I)$ with Hatcher--Wagoner's stable pseudo-isotopy theory $P(M\times I)=\mathrm{colim}\,C(M\times I^N)$ (\cite{HW}, see also \cite{Ig1}), where the colimit is taken with respect to the ``suspension functor'' $\sigma\colon C(X)\to C(X\times I)$ (\cite[Ch I-\S{5}]{HW}). In \cite{Ig2}, the following commutative diagram is considered for $\dim{X}=4$:
\[ \xymatrix{
  \pi_0 \calD(X) \ar[r]^-{\widetilde{\theta}} \ar[d]_-{\overline{\lambda}} & \pi_0 C(X) \ar[d] & &\\
  \mathit{Wh}_1^+(\pi_1X;\Z_2\oplus \pi_2X) \ar[r] & \pi_0 P(X) \ar[r] & \mathit{Wh}_2(\pi_1X) \ar[r] & 0
}\]
where the bottom horizontal line is exact (by \cite{HW}) and $\calD(X)$ is the space of ``lens-space models'' for pseudo-isotopies of $X$ (\cite[Definition~1.4]{Ig2}). In recent works of K.~Igusa (\cite{Ig1,Ig2}) and O.~Singh (\cite{Si}), many nontrivial elements of $\pi_0C(X)$ for some 4-manifolds $X$ with $\pi_2X\neq 0$ are found, by realizing elements of $\mathit{Wh}_1^+(\pi_1X;\Z_2\oplus \pi_2X)$ by explicit 1-parameter families in $\calD(X)$ of 2,3-handle pairs. Their nontrivial elements are also nontrivial in $\pi_0 P(X)$. 

We see that our theta-graph surgery behaves differently. 
\begin{Thm}\label{thm:2}
Let $M$ be a spherical 3-manifold.
\begin{enumerate}
\item Theta-graph surgery gives elements of $\pi_0C(M\times I)$ that lift to $\pi_0\calD(M\times I)$.

\item $\pi_0C(M\times I)$ includes a free abelian subgroup of rank $\dim\calA_\Theta^\even(\mathrm{Ker}\,\ve)$ that is included in the kernel of $\pi_0C(M\times I)\to \pi_0 P(M\times I)$.
\end{enumerate}
\end{Thm}
\begin{proof}
(1) follows from \cite[\S{8}]{Wa20}, which uses a result of \cite{BW}. Then (2) follows since $\pi_2(M\times I)=\pi_2M=0$, $\pi_1(M\times I)=\pi_1M$, and 
\[ \mathit{Wh}_1^+(\pi_1M;\Z_2)=\bigoplus_{c-1}\Z_2, \]
where $c$ is the number of conjugacy classes of elements of $\pi_1M$ (\cite[p.3]{Ig2}).
\end{proof}

\begin{Rem}
The subgroup of Theorem~\ref{thm:2} (2) is of finite index ($=2^{c-1}$) in the free abelian subgroup of rank $\dim\calA_\Theta^\even(\mathrm{Ker}\,\ve)$ generated by theta-graph surgery. This restriction to the smaller subgroup would be unnecessary since Theorem~\ref{thm:2} (2) could also be proved by directly evaluating the homomorphism $\overline{\lambda}\colon \pi_0 \calD(X)\to \mathit{Wh}_1^+(\pi_1X;\Z_2)$ for the 1-parameter family of the attaching 2-spheres of the 3-handles in \cite[\S{8}]{Wa20} obtained from theta-graph surgery. 

There is a similar result for the space $C^{\mathrm{Top}}(M)$ of topological pseudo-isotopies by Kwasik and Schultz \cite[Theorem~1]{KS} giving nontrivial unstable elements of $\pi_0C^{\mathrm{Top}}(M)$ for certain 3-manifolds $M$.
\end{Rem}

Let $\widetilde{\Diff}(M)$ be the space of block diffeomorphisms of $M$ (see e.g., \cite{HLLRW}). In \cite[Propopsition~2.1]{Hat3}, Hatcher constructed a spectral sequence with $E_{pq}^1=\pi_qC(M\times I^p)$ converging to $\pi_{p+q+1}(\widetilde{\Diff}(M)/\Diff(M))$. In particular,
\[ \pi_1(\widetilde{\Diff}(M)/\Diff(M))=E^2_{00}=E_{00}^1/\delta_*(E_{10}^1)=\pi_0C(M)/\delta_*(\pi_0C(M\times I)), \]
where $\delta\colon C(M\times I)\to C(M)$ is defined by $\delta(f)=f|_{M\times I\times\{1\}}$. Similar identity for the topological case was considered in \cite[p.874]{KS}. 
Since the elements of $\pi_0C(M)$ constructed by theta-graph surgery are in the image of $\delta_*$ (\cite[Theorem~1.3]{BW}), we have the following.
\begin{Prop}
Let $M$ be a spherical 3-manifold. 
The elements of $\pi_0C(M)$ constructed by theta-graph surgery are trivial in $\pi_1(\widetilde{\Diff}(M)/\Diff(M))$.
\end{Prop}
\sloppy
This shows that theta-graph surgery is not like the unstable elements of $\pi_0C^{\mathrm{Top}}(M)$ detected in \cite{KS}, where the nontrivial elements are detected in $\pi_1(\widetilde{\mathrm{Top}}(M)/\mathrm{Top}(M))$.

\subsection*{\bf Acknowledgments}

We thank Y.~Nozaki, T.~Sakasai for information about software. We thank E.~Matsuhashi for his support. %I thank Y.~Ohta for his computation in Lemma~\ref{lem:ohta}.
This work was partially supported by JSPS Grant-in-Aid for Scientific Research 20K03594 and 21K03225, and by RIMS, Kyoto University.

\fussy
%%%%%%%%%%%%%%%%%%%%%%%%%%%%%%%
%%%%%%%%%%%%%%%%%%%%%%%%%%%%%%%
\mysection{Twisted homology of configuration space}{s:homol}

In this and the next section, we check Theorem~\ref{thm:1}.
We use a slightly generalized version of the invariant $Z_\Theta^\even$ of \cite{Wa20} for more general local coefficient system $W$ (Lemma~\ref{lem:ex-assum}). To define $Z_\Theta^\even$, we need ``propagators'' in a family of configuration spaces of two points on $X=M\times I$. Now we check that a propagator exists for more general coefficients $W$.

%%%%%%%%%%%%%%%%%%%%%%%%%%%%%%%
\subsection{Acyclic complex}

Suppose that $M$ is a spherical 3-manifold. 
Let $\pi=\pi_1M$, and let $A$ be a non-trivial irreducible $\C\pi$-module. 
The homology of $M$ with twisted coefficient $A$ is defined by
\[ H_*(M;A):=H(S_*(\widetilde{M})\otimes_{\C\pi}A), \]
where $\widetilde{M}$ is the universal cover of $M$, $S_*(\cdot)$ is the $\C$-complex of singular chains.

\begin{Lem}\label{lem:M-acyclic}
$H_*(M;A)=0$. 
\end{Lem}
\begin{proof}
Since $\pi$ is finite, $\C\pi$ is semisimple in the sense of \cite[\S{I.4}]{CE} by Maschke's theorem and we have $H^1(\pi;A)=0$ (Theorem~VI.16.6 and Lemma~VI.16.7 of \cite{HS}). By the universal coefficient theorem, which is valid if the ring is hereditary (e.g., $\C\pi$ for $\pi$ finite), the sequence
\[ 0\to \mathrm{Ext}_{\C\pi}^1(H_{i-1}(C),A)\to H^i(\mathrm{Hom}_{\C\pi}(C,A))\to \mathrm{Hom}_{\C\pi}(H_i(C),A)\to 0 \]
is exact for $C=S_*(S^3;\C)$ (as a $\C\pi$-module) and any $\C\pi$-module $A$ (e.g., \cite[Theorem~VI.3.3]{CE}). Hence we have
\[ \begin{split}
&H^3(M;A)\cong \mathrm{Hom}_{\C\pi}(\C,A)\cong H^0(\pi;A)=A^{\pi},\\
&H^2(M;A)=0,\\
&H^1(M;A)\cong\mathrm{Ext}_{\C\pi}^1(\C,A)=H^1(\pi;A)=0,\\
&H^0(M;A)\cong\mathrm{Hom}_{\C\pi}(\C,A)\cong H^0(\pi;A)=A^{\pi}.
\end{split} \]
Since $A$ is a non-trivial irreducible $\C\pi$-module, we have $A^\pi=0$. 
Then by Poincar\'{e} duality (e.g., \cite[\S{3.H}]{Hat2}, \cite[Theorem~7.17]{Hatt} etc.), we also have $H_*(M;A)=0$. 
\end{proof}

%%%%%%%%%%%%%%%%%%%%%%%%%%%%%%%
\subsection{Propagator for spherical 3-manifolds $M$}

Let $\Delta_M$ be the diagonal of $M\times M$. The configuration space of two points of $M$ is
\[ \Conf_2(M)=M\times M-\Delta_M. \]
The Fulton--MacPherson compactification of $\Conf_2(M)$ is
\[ \bConf_2(M)=B\ell_{\Delta_M}(M\times M). \]
We idenfity the boundary $\partial\bConf_2(M)$, which is the unit normal sphere bundle of $\Delta_M\subset M\times M$, with $ST(M)$, the unit tangent sphere bundle.
We make the following assumption.
\begin{Assum}\label{assum:2-acyclic}
A $\pi\times \pi$-module $W$ satisfies the following conditions:
\begin{enumerate}
\item $H_*(M\times M;W)=0$.
\item There are elements $e_W^1,\ldots,e_W^r\in W$ on which $\pi\times \pi$ acts trivially such that $H_*(\Delta_M;W)\cong H_*(M;\C)^{\oplus r}$, which is generated over $\C$ by $[*\otimes e_W^i]$ and $[M\otimes e_W^i]$ ($i=1,\ldots,r$). 
\end{enumerate}
\end{Assum}
For a non-trivial irreducible $\pi$-module $A$, let $A\boxtimes A^*$ denote the pullback of the local coefficient system $A\boxtimes_\C A^*$ on $M\times M$ to $\bConf_2(M)$. We denote by $A\otimes A^*$ the restriction of $A\boxtimes A^*$ to $\partial\bConf_2(M)$, on which $\pi$ acts diagonally. 
\begin{Lem}\label{lem:ex-assum}
The following $\pi\times\pi$-modules $W$ satisfy Assumption~\ref{assum:2-acyclic}.
\begin{enumerate}
\item $W=A\boxtimes A^*$.
\item $W=\bigoplus_i (A_i\boxtimes A_i^*)$, where $A_i$ is a non-trivial irreducible $\pi$-module.
\item $W=\mathrm{Ker}\,\ve$, where $\ve\colon \C\pi\to \C$ is the augmentation map.
\end{enumerate}
\end{Lem} 
\begin{proof}
That (1) satisfies Assumption~\ref{assum:2-acyclic} follows from Lemma~\ref{lem:M-acyclic} and the K\"{u}nneth formula for $\C\pi$-modules. The case (2) follows from (1) and
\[ \begin{split}
  &H_*(M\times M;\textstyle\bigoplus_i (A_i\boxtimes A_i^*))=\bigoplus_i H_*(M\times M;A_i\boxtimes A_i^*),\\
  &H_*(\Delta_M;\textstyle\bigoplus_i (A_i\otimes A_i^*))=\bigoplus_i H_*(\Delta_M;A_i\otimes A_i^*).
\end{split} \]
The case (3) follows from the $\pi\times\pi$-module decomposition
\begin{equation}\label{eq:cpi}
 \C\pi\cong\bigoplus_{i}\mathrm{End}(A_i),\quad
 \mathrm{Ker}\,\ve\cong \bigoplus_{i\neq [1]}\mathrm{End}(A_i)
 \cong \bigoplus_{i\neq [1]}(A_i\boxtimes A_i^*),
\end{equation}
(e.g., \cite[Proposition~3.29]{FH}) where $i$ is taken over the conjugacy classes in $\pi$, and from (2).
\end{proof}

\begin{Lem}\label{lem:H(Conf)}
Let $W$ be a $\pi\times\pi$-module satisfying Assumption~\ref{assum:2-acyclic}. Then we have
\[ H_k(\bConf_2(M);W)\cong\left\{\begin{array}{ll}
\C\{[ST(*)\otimes e_W^i]\mid i=1,\ldots,r\} & (k=2),\\
\C\{[ST(M)\otimes e_W^i]\mid i=1,\ldots,r\} & (k=5),\\
0 & (\mbox{otherwise}),
\end{array}\right. \]
where for an oriented submanifold $\sigma$ of $M$, we denote by $ST(\sigma)$ the restriction of the unit tangent 2-sphere bundle $ST(M)$ to $\sigma$.
\end{Lem}
\begin{proof}
We consider the exact sequence
\[
 H_{i+1}(M^{\times 2};W)\to H_{i+1}(M^{\times 2},\Conf_2(M);W)\to H_i(\Conf_2(M);W)\to H_i(M^{\times 2};W),
\]
where $H_*(M^{\times 2};W)=0$ by Assumption~\ref{assum:2-acyclic} (1). 
Letting $N(\Delta_M)$ be a closed tubular neighborhood of $\Delta_M$, we have
\[ H_{i+1}(M^{\times 2},\Conf_2(M);W)\cong H_{i+1}(N(\Delta_M),\partial N(\Delta_M);W) \]
by excision. Since $M$ is parallelizable and the normal bundle of $\Delta_M$ can be canonically identified with $TM$, the normal bundle of $\Delta_M$ is trivial. By Assumption~\ref{assum:2-acyclic} (2), we have
\[\begin{split}
& H_{i+1}(N(\Delta_M),\partial N(\Delta_M);W)
=H_3(D^3,\partial D^3;\C)\otimes_\C H_{i-2}(\Delta_M;W)\\
&\cong H_3(D^3,\partial D^3;\C)\otimes_\C H_{i-2}(M;\C)^{\oplus r}\cong H_{i-2}(M;\C)^{\oplus r}.
\end{split} \]
Here, $H_{i-2}(M;\C)$ is rank 1 for $i-2=0,3$, and its generator is $*,M$, respectively.
\end{proof}

\begin{Lem}
Let $W$ be a $\pi\times\pi$-module satisfying Assumption~\ref{assum:2-acyclic}. Let $s_{\tau_0}\colon M\to ST(M)$ be the section given by the normalization of the first vector of a framing $\tau_0$ of $M$. Then we have
\[ \begin{split}
	H_3(\partial \bConf_2(M);W)=\C\{[s_{\tau_0}(M)\otimes e_W^i]\mid i=1,\ldots,r\}.
\end{split} \]
\end{Lem}
\begin{proof}
This follows from the trivialization $\partial \bConf_2(M)\cong S^2\times M$ induced by $\tau_0$, Assumption~\ref{assum:2-acyclic} (2), and the K\"{u}nneth formula for $\C$-modules.
\end{proof}

\begin{Lem}[Propagator]\label{lem:propagator} Let $W$ be a $\pi\times\pi$-module satisfying Assumption~\ref{assum:2-acyclic}. 
\begin{enumerate}
\item There exists a 4-chain $\omega^i$ of $\bConf_2(M)$ with coefficients in $W$ that is transversal to the boundary and that satisfies
\[ \partial_W \omega^i=s_{\tau_0}(M)\otimes e_W^i. \]

\item For a fixed framing $\tau_0$ and $i$, the extension $\omega^i$ is unique in the sense that for two such extensions $\omega^i$ and $\omega'{}^i$ that agree near the boundary, there is a 5-chain $\eta$ of $\mathrm{Int}\,\bConf_2(M)$ with coefficients in $W$ such that
\[ \omega'{}^i - \omega^i = \partial_W \eta. \]
\end{enumerate}
We call the direct sum of such extensions $\omega=\sum_i\omega^i$ a \emph{propagator} for $\tau_0$.
\end{Lem}
\begin{proof}
The assertion (1) follows immediately from the long exact sequence
\[ 
  H_4(\bC;W)\to H_4(\bC,\partial\bC;W)\to H_3(\partial\bC;W)\to H_3(\bC;W),
\]
where we abbreviate as $\bC=\bConf_2(M)$, and $H_4(\bC;W)=H_3(\bC;W)=0$ by Lemma~\ref{lem:H(Conf)}. Here, both $[\omega^i]$ and $[s_{\tau_0}(M)\otimes e_W^i]$ restrict to the same generator of the homology of $*\times S^2\subset SN\Delta_{M}$, their homology classes agree. The assertion (2) follows since the difference $\omega'{}^i-\omega^i$ vanishes near $\partial\bC$ and represents 0 of the twisted homology $H_4(\bC;W)$.
\end{proof}

%%%%%%%%%%%%%%%%%%%%%%%%%%%%%%%
\subsection{Twisted homologies of the configuration space of $M\times I$ and its family}

Let $\bConf_2(M\times I)=B\ell_{\Delta_{M\times I}}((M\times I)\times (M\times I))$, which is not a smooth manifold with corners, but satisfies the Poincar\'{e}--Lefschetz duality (\cite[C.3]{Wa21}).
For an $(M\times I)$-bundle $p:E\to S^1$ over $S^1$ with structure group $\Diff(M\times I,\partial)$, we denote by
\[ \bConf_2(p)\colon E\bConf_2(p)\to S^1 \]
the associated $\bConf_2(M\times I)$-bundle with structure group $\Diff(M\times I,\partial)$. 

To define the invariant $Z_\Theta^\even$ as in \cite[Proposition~7.1 and Remark~7.2]{Wa20}, we need a {\it propagator in family}, which is a 6-chain of $E\bConf_2(p)$ with coefficients in $W$ satisfying some boundary condition similar to Lemma~\ref{lem:propagator}, and is implicitly defined in the proof of \cite[Proposition~7.1]{Wa20}. In \cite{Wa20}, the existence of such a 6-chain was guaranteed by the lemmas \cite[Lemmas~7.3 and 7.4]{Wa20}. The analogues of the lemmas for the $W$ in this paper are the following, whose proofs are the same except that the invariant element $c_A$ (\cite[Assumption~3.6]{Wa20}) is replaced with $\sum_i e_W^i$.

\begin{Lem}[{\cite[Lemma~7.3]{Wa20}}]\label{lem:H(Conf(MxI)}
Let $W$ be a $\pi\times\pi$-module satisfying Assumption~\ref{assum:2-acyclic}. Then we have
\[ H_k(\bConf_2(M\times I);W)\cong\left\{\begin{array}{ll}
\C\{[ST(*)\otimes e_W^i]\mid i=1,\ldots,r\} & (k=3),\\
\C\{[ST(M)\otimes e_W^i]\mid i=1,\ldots,r\} & (k=6),\\
0 & (\mbox{otherwise}),
\end{array}\right. \]
where we identify $M$ with $M\times\{\frac{1}{2}\}$ in $M\times I$, and for an oriented submanifold $\sigma$ of $M\times I$, we denote by $ST(\sigma)$ the restriction of the unit tangent 3-sphere bundle $ST(M\times I)$ to $\sigma$.
\end{Lem}

\begin{Lem}[{\cite[Lemma~7.4]{Wa20}}]\label{lem:5-cycle-extend}
Let $W$ be a $\pi\times\pi$-module satisfying Assumption~\ref{assum:2-acyclic}. Then we have
$H_5(E\bConf_2(p);W)=0$ and 
the natural map
\[ H_6(E\bConf_2(p);W)
\to H_6(E\bConf_2(p),\partial E\bConf_2(p);W) \]
is zero. Thus the connecting homomorphism
\[ H_6(E\bConf_2(p),\partial E\bConf_2(p);W)
\to H_5(\partial E\bConf_2(p);W) \]
is an isomorphism.
\end{Lem}

Roughly, a propagator in family $E\bConf_2(p)$ is constructed as follows. The boundary of $\bConf_2(X)$ is $p_{B\ell}^{-1}(\Delta_X\cup (\partial X\times X)\cup (X\times\partial X))$, where $p_{B\ell}\colon \bConf_2(X)\to X\times X$ is the canonical blow-down projection. 
\begin{itemize}
\item On the stratum of $\partial E\bConf_2(p)$ corresponding to the part $p_{B\ell}^{-1}(\Delta_X)$, we take the 5-chain $s_\tau(E)\otimes e_W^i$, where $s_\tau\colon E\to ST^vE$ is the section given by the normalization of the first vector of a vertical framing $\tau\colon T^vE=\mathrm{Ker}\,dp\stackrel{\cong}{\to} \R^4\times E$. 

\item On the stratum of $\partial E\bConf_2(p)$ corresponding to the part $p_{B\ell}^{-1}((\partial X\times X)\cup (X\times \partial X))$, we take the pullbacks of the copies of $\omega^i$ of Lemma~\ref{lem:propagator} in the subspace $p_{B\ell}^{-1}((M\times \{0\})\times (M\times \{0\}))\cup p_{B\ell}^{-1}((M\times \{1\})\times (M\times \{1\}))\cong \bConf_2(M)\coprod\bConf_2(M)$. This makes sense since the bundle $p$ is trivialized over $\partial X$. 
\end{itemize}
The sum of these 5-chains of $\partial E\bConf_2(p)$ is a cycle. Then by Lemma~\ref{lem:5-cycle-extend}, it has an extension to a 6-chain $\widetilde{\omega}^i$ of $E\bConf_2(p)$. We call $\widetilde{\omega}=\sum_i\widetilde{\omega}^i$ a propagator in family.

%%%%%%%%%%%%%%%%%%%%%%%%%%%%%%%
%%%%%%%%%%%%%%%%%%%%%%%%%%%%%%%
\mysection{Properties of the invariant $Z_\Theta^\even$}{s:Z}

Roughly, the invariant $Z_\Theta^\even$ is defined by choosing three propagators $\widetilde{\omega}_1,\widetilde{\omega}_2,\widetilde{\omega}_3$ in family $E\bConf_2(p)$ with some boundary conditions and then by 
\[ Z_\Theta^\even(\widetilde{\omega}_1,\widetilde{\omega}_2,\widetilde{\omega}_3)=\frac{1}{6}\Tr_\Theta\langle\widetilde{\omega}_1,\widetilde{\omega}_2,\widetilde{\omega}_3\rangle_\Theta\in \calA_\Theta^\even(W), \]
where $\Tr_\Theta\colon W^{\otimes 3}\to \calA_\Theta^\even(W)$ is the projection, $\langle -,-,-\rangle_\Theta$ is the triple intersection among chains with twisted coefficients (\cite[\S{5}]{Wa20}). We will not repeat the detailed definition here. The only difference in the proof of the well-definedness of $Z_\Theta^\even$ is to replace the invariant element $c_A$ in \cite[Proofs of Theorem~5.3 and Propositon~7.1]{Wa20} with $\sum_i e_W^i$. 
The property we need is the following.

\begin{Prop}[{\cite[Theorem~6.2]{Wa20}}]\label{prop:surgery}
Let $X,M,\pi,W$ be as in Lemma~\ref{lem:ex-assum}. 
Then for any $g_1,g_2,g_3\in\pi$, an element $\Psi_1(\Theta(g_1,g_2,g_3))$ of $\Omega_1^{SO}(B\Diff_0(X,\partial))$ is defined by surgery on an embedded theta-graph associated to $(g_1,g_2,g_3)$, which belongs to the image from $\pi_1B\Diff_0(X,\partial)$, and the following identity holds.
\[ Z_\Theta^\even(\Psi_1(\Theta(g_1,g_2,g_3)))=2\,[\rho_W(g_1)\wedge\rho_W(g_2)\wedge\rho_W(g_3)], \]
where $\rho_W\colon\C[\pi]\to W=\bigoplus_i\End(A_i)$ is the representation of $\pi$.
\end{Prop}

Note that the invariant $Z_\Theta^\even$ in \cite{Wa20} was defined on a slightly different group\footnote{This should not be confused with the classifying space of $\widetilde{\Diff}(M\times I,\partial)$.} $\pi_1 \widetilde{B\Diff}_{\mathrm{deg}}(M\times I,\partial)$ than $\pi_1B\Diff_0(X,\partial)$, however, it can be shown that $Z_\Theta^\even$ descends to a map from the image of $\pi_1 \widetilde{B\Diff}_{\mathrm{deg}}(M\times I,\partial)\to \pi_1B\Diff_0(M\times I,\partial)$ (Lemma~\ref{lem:Z-descends} below). Namely, the homotopy fiber of the natural map $\widetilde{B\Diff}_{\mathrm{deg}}(M\times I,\partial)\to B\Diff_0(M\times I,\partial)$ is given by 
\[ \Map((M\times I,\partial), (SO_4,1))\times \Map((M\times I,\partial),(M\times I,\partial))_{\mathrm{id}}, \]
where $\partial=\partial(M\times I)$, $\Map((A,\partial A),(C,D))$ denotes the space of continuous maps $(A,\partial A)\to (C,D)$ with the $C^0$-topology that agree with the base map on $\partial A$, $(-)_{\mathrm{id}}$ denotes the component of the identity. The first factor $\Map((M\times I,\partial), (SO_4,1))$ is identified with the space of framings on $M\times I$, the second factor $\Map((M\times I,\partial),(M\times I,\partial))_{\mathrm{id}}$ is used to give a fiberwise (relative) degree one map $(E,\partial E)\to (M\times I, \partial)$ to pullback a local coefficient system on $M$. 

\begin{Lem}
For a spherical 3-manifold $M$,
\[\begin{split}
&\Map((M\times I,\partial), (SO_4,1))\times \Map((M\times I,\partial),(M\times I,\partial))_{\mathrm{id}}\\
&\simeq \Omega\Map(M,SO_4)\times \Omega\Map(M,M)_{\mathrm{id}}.
\end{split} \]
Furthermore, there is a fibration sequence:
\[
\Omega^4SO_4\longrightarrow 
\Omega\Map(M,SO_4)\longrightarrow
\Omega SO_4. \]
\end{Lem}
\begin{proof} We have the following homotopy equivalences:
\[ \begin{split}
&\Map((M\times I,\partial), (SO_4,1))\simeq \Omega\Map(M,SO_4),\\
&\Map(M\times I,\partial),(M\times I,\partial))_{\mathrm{id}}\simeq \Omega\Map(M,M)_{\mathrm{id}},
\end{split} \]
where the basepoints of $\Map(M,SO_4)$ and $\Map(M,M)_{\mathrm{id}}$ are the constant map to 1 and $\mathrm{id}$, respectively. Furthermore, we have the following fibration sequence:
\[
\Map_*(M,SO_4)\to \Map(M,SO_4)\stackrel{\mathrm{ev}}{\to} SO_4,
\]
where $\mathrm{ev}$ is the evaluation at a fixed basepoint of $M$, and $\Map_*(-,-)$ is the subspace of $\Map(-,-)$ of pointed maps. Then the result follows by the homotopy equivalence:
$\Omega\Map_*(M,SO_4)\simeq \Map_*(S^1\wedge M,SO_4)\simeq \Map_*(S^4,SO_4)$.
\end{proof}
\begin{Lem}\label{lem:Z-descends}
Let $\Gamma_0(M\times I)$ denote the image of the natural map $\pi_1\widetilde{B\Diff}_{\mathrm{\deg}}(M\times I,\partial)\to \pi_1B\Diff_0(M\times I)$. 
The homomorphism $Z_\Theta^\even\colon \pi_1\widetilde{B\Diff}_{\mathrm{\deg}}(M\times I,\partial)\to \calA_\Theta^\even(W)$ descends to a map 
$\Gamma_0(M\times I)\to \calA_\Theta^\even(W)$.
\end{Lem}
\begin{proof}

Since $\pi_1\Omega SO_4=0$ and $\pi_1\Omega^4 SO_4=\pi_5SO_4=\Z_2\oplus \Z_2$, $\pi_1\Omega\Map(M,SO_4)=\pi_2\Map(M,SO_4)$ is a quotient of $\Z_2\oplus \Z_2$. Moreover, a change of the choice of the lift of an element of $\Gamma_0(M\times I)$ to $\pi_1\widetilde{B\Diff}_{\mathrm{\deg}}(M\times I,\partial)$ within the factor $\Map((M\times I,\partial),(M\times I,\partial))_{\mathrm{id}}$ in the homotopy fiber does not affect the value of $Z_\Theta^\even$ since the local coefficient system on the total space of the corresponding $(M\times I)$-bundle $p\colon E\to S^1$, which is needed to define $Z_\Theta^\even$, is determined by the homotopy class of the induced map $\pi_1E\to \pi_1M$, which is canonically fixed since we have the canonical decomposition $\pi_1E=\pi_1 S^1\times \pi_1 M$ by the van Kampen theorem and the homotopical triviality of elements of $\Diff_0(M\times I,\partial)$, and the map $\pi_1E\to \pi_1M$ is just the projection to the second factor.
This completes the proof.
\end{proof}

Since $\rho_W\colon\C[\pi]\to W=\bigoplus_i\End(A_i)$ in Proposition~\ref{prop:surgery} is surjective for $W=\mathrm{Ker}\,\ve$ (see (\ref{eq:cpi})), we have the following.
\begin{Cor}
The homomorphism $Z_\Theta^\even\colon \Gamma_0(M\times I)\otimes\C\to \calA_\Theta^\even(\mathrm{Ker}\,\ve)$ is surjective.
\end{Cor}

%%%%%%%%%%%%%%%%%%%%%%%%%%%%%%
%%%%%%%%%%%%%%%%%%%%%%%%%%%%%%
\mysection{Example 1: Poincar\'{e} homology sphere}{s:upper}

%%%%%%%%%%%%%%%%%%%%%%%%%%%%%%
\subsection{The group $\mathrm{SL}_2(\F_5)$}

Let $\hat{\pi}$ denote the set of conjugacy classes of $\pi=\mathrm{SL}_2(\F_5)$. It is known that $\hat{\pi}$ has 9 elements, represented respectively by the following elements:
\[\begin{split}
 &\pm I=\left(\begin{array}{cc}
  \pm 1 & 0\\
  0 & \pm 1
\end{array}\right),\quad
\alpha=\left(\begin{array}{cc}
  2 & 0\\
  0 & 3
\end{array}\right),\quad
\beta=\left(\begin{array}{cc}
  0 & -1\\
  1 & -1
\end{array}\right),\quad
\beta'=\left(\begin{array}{cc}
  0 & -1\\
  1 & 1
\end{array}\right),\\
&\gamma=\left(\begin{array}{cc}
  1 & 1\\
  0 & 1
\end{array}\right),\quad
\gamma'=\left(\begin{array}{cc}
  1 & 2\\
  0 & 1
\end{array}\right),\quad
-\gamma=\left(\begin{array}{cc}
  -1 & -1\\
  0 & -1
\end{array}\right),\quad
-\gamma'=\left(\begin{array}{cc}
  -1 & -2\\
  0 & -1
\end{array}\right).
\end{split} \]

We give a list of all elements of $\mathrm{SL}_2(\F_5)$ in Appendix~\ref{s:elements}.

\begin{Lem}\label{lem:involution-inv}
For $\pi=\mathrm{SL}_2(\F_5)$, $\hat{\pi}$ is invariant under taking the inverse. Namely, for each class $[x]\in\hat{\pi}$, we have $[x^{-1}]=[x]$.
\end{Lem}
\begin{proof}
It suffices to see that the inverse of each element $x$ in the above list of 9 elements is conjugate to $x$, which can be checked by comparing the Jordan canonical forms, or from the list in Appendix~\ref{s:elements} (obtained by using the Jordan canonical forms). Note that we only allow the conjugation $gxg^{-1}$ by $g\in \mathrm{SL}_2(\F_5)$.
\end{proof}

%%%%%%%%%%%%%%%%%%%%%%%%%%%%%%
\subsection{Representation of $\mathrm{SL}_2(\F_5)$}

There are 9 distinct irreducible represenations $A_i$ ($i=1,2,\ldots,9$) of the group $\pi$ whose character is given as in Table~\ref{tab:ch}, and any irreducible representation of $\pi$ over $\C$ is isomorphic to one of them. Irreducible representations of $\pi\times \pi$ are given by the external tensor product $A_i\boxtimes A_j$ (e.g., \cite[Exercise~2.36]{FH}). Since the values of the characters are real for $\pi=\mathrm{SL}_2(\F_5)$, we have
\begin{equation}\label{eq:C_pi} \C[\pi]\cong \bigoplus_{i=1}^9\mathrm{End}(A_i)\cong \bigoplus_{i=1}^9 (A_i\boxtimes A_i),\quad 
\mathrm{Ker}\,\ve\cong \bigoplus_{i=2}^9 (A_i\boxtimes A_i)
\end{equation}
as $\pi\times \pi$-modules\footnote{If the characters are not real, we have $\mathrm{End}(A_i)\cong A_i\boxtimes A_i^*$.}, where the $\pi\times\pi$-invariant $A_1\boxtimes A_1$ in $\C[\pi]$ corresponds to the subspace spanned by the element $\sum_{g\in \pi}g\in\C[\pi]$, and $\mathrm{Ker}\,\ve$ is a $\pi\times\pi$-submodule of $\C[\pi]$.

\par\medskip
\begin{table}
\centering
  \begin{tabular}{|c|ccccccccc|}  \hline
    $\hat{\pi}$ & $I$ & $-I$ & $\alpha$ & $\beta$ & $\beta'$ & $\gamma$ & $\gamma'$ & $-\gamma$ & $-\gamma'$ \\ 
    size & 1 & 1 & 30 & 20 & 20 & 12 & 12 & 12 & 12 \\ \hline
    $A_1$ & 1 & 1 & 1 & 1 & 1 & 1 & 1 & 1 & 1 \\ 
    $A_2$ & 2 & $-2$ & 0 & $-1$ & 1 & $-\phi^*$ & $-\phi$ & $\phi^*$ & $\phi$ \\ 
    $A_3$ & 2 & $-2$ & 0 & $-1$ & 1 & $-\phi$ & $-\phi^*$ & $\phi$ & $\phi^*$ \\ 
    $A_4$ & 3 & 3 & $-1$ & 0 & 0 & $\phi$ & $\phi^*$ & $\phi$ & $\phi^*$ \\ 
    $A_5$ & 3 & 3 & $-1$ & 0 & 0 & $\phi^*$ & $\phi$ & $\phi^*$ & $\phi$ \\ 
    $A_6$ & 4 & 4 & 0 & 1 & 1 & $-1$ & $-1$ & $-1$ & $-1$ \\ 
    $A_7$ & 4 & $-4$ & 0 & 1 & $-1$ & $-1$ & $-1$ & 1 & 1 \\ 
    $A_8$ & 5 & 5 & 1 & $-1$ & $-1$ & 0 & 0 & 0 & 0 \\ 
    $A_9$ & 6 & $-6$ & 0 & 0 & 0 & 1 & 1 & $-1$ & $-1$ \\ \hline
  \end{tabular}
\par\medskip
\caption{The characters $\rho_{A_i}(g)$ for $\mathrm{SL}_2(\F_5)$. $\phi=\cfrac{1+\sqrt{5}}{2}$, $\phi^*=\cfrac{1-\sqrt{5}}{2}$. The values of the characters are taken from \cite{FH,CP}. The order of the rows and columns followed the one in \cite{Bo}, although the original table in \cite{Bo} includes few typos in the row of $A_7$ (signs of the last four entries).}\label{tab:ch}
\end{table}

%%%%%%%%%%%%%%%%%%%%%%%%%%%%%%
\subsection{Computation of the character}

To get the dimension of $\calA_\Theta^{\odd/\even}(\C[\pi])$, we compute the dimensions of the invariants $(\tbigwedge^3 W)^{(\pi\times \pi)\rtimes \Z_2}$ and $(\Sym^3 W)^{(\pi\times \pi)\rtimes \Z_2}$ for $W=\bigoplus_{i=1}^9 (A_i\boxtimes A_i)$. Here, recall that the semidirect product structure on $(\pi\times \pi)\rtimes \Z_2$ is given by the homomorphism $\psi\colon\Z_2=\{1,\tau\}\to \mathrm{Aut}(\pi\times \pi)$; $\tau\mapsto ((x,y)\mapsto (y,x))$. This is suitable since we have
\[ \begin{split}
  &\tau\cdot(x,y)\cdot \tau^{-1}(g\wedge h\wedge k)=\tau\cdot(xg^{-1}y^{-1}\wedge xh^{-1}y^{-1}\wedge xk^{-1}y^{-1})\\
  &=ygx^{-1}\wedge yhx^{-1}\wedge ykx^{-1}=(y,x)(g\wedge h\wedge k),
\end{split} \]
which shows that the given actions of $\pi\times\pi$ and $\Z_2$ on $\tbigwedge^3\C[\pi]$ agrees with that of the semidirect product, and similarly for $\Sym^3\,\C[\pi]$.

\begin{Lem}%[{Y.~Ohta (\cite{Oh})}]
\label{lem:ohta}
For $\pi=\mathrm{SL}_2(\F_5)$, we have the following.
\begin{enumerate}
\item $\dim\,(\tbigwedge^3 \C[\pi])^{(\pi\times \pi)\rtimes \Z_2}=27$.
\item $\dim\,(\Sym^3\, \C[\pi])^{(\pi\times \pi)\rtimes \Z_2}=65$.
\end{enumerate}
\end{Lem}
\begin{proof}
Recall that the dimension $m$ of the invariant part in a representation $V$ of a finite group $G$ can be given by the following formula (\cite[(2.9)]{FH}):
\[ m=\frac{1}{|G|}\sum_{g\in G}\mathrm{Trace}(g|_V)=\frac{1}{|G|}\sum_{g\in G}\chi_V(g). \]
We apply this formula for $G=(\pi\times \pi)\rtimes \Z_2$ with the formulas: $\chi_{V\oplus V'}(g)=\chi_V(g)+\chi_{V'}(g)$, $\chi_{V\boxtimes V'}(g,h)=\chi_{V}(g)\chi_{V'}(h)$, and
\[ \begin{split}
 &\chi_{\bigwedge^3(V)}(g)=\cfrac{1}{6}( \chi_V(g)^3-3 \chi_V(g^2) \chi_V(g)+2 \chi_V(g^3)),\\
  &\chi_{\mathrm{Sym}^3(V)}(g)=\cfrac{1}{6}\left( \chi_V(g)^3+3 \chi_V(g^2)  \chi_V(g) +2 \chi_V(g^3)\right),\\
\end{split}\]
with the character table (Table~\ref{tab:ch}).
First, for $W=\C[\pi]$ we compute
\[ \begin{split}
  \dim\,(\tbigwedge^3 W)^{\pi\times \pi}
  &=\displaystyle\frac{1}{6|\pi|^2}\sum_{g,h\in\pi}(\chi_W(g,h)^3-3 \chi_W(g^2,h^2) \chi_W(g,h)+2 \chi_W(g^3,h^3)),\\
  \dim\,(\Sym^3 W)^{\pi\times \pi}
  &=\displaystyle\frac{1}{6|\pi|^2}\sum_{g,h\in\pi}(\chi_W(g,h)^3+3 \chi_W(g^2,h^2) \chi_W(g,h)+2 \chi_W(g^3,h^3)),
\end{split} \]
where $\chi_W(g,h)=\sum_{i=1}^9\chi_{A_i}(g)\chi_{A_i}(h)$, $\chi_W(g^2,h^2)=\sum_{i=1}^9\chi_{A_i}(g^2)\chi_{A_i}(h^2)$, and so on. Substituting the values of the characters of Table~\ref{tab:ch} into these formulas, we get 
\[ \begin{split}
&\dim\,(\tbigwedge^3 \C[\pi])^{\pi\times \pi}=33,\quad \dim\,(\Sym^3\, \C[\pi])^{\pi\times\pi}=71.
\end{split} \]
The detail of these computations by Maxima can be found in \cite{Oh}, in which Table~\ref{tab:conj} is used to compute the characters of $g^2$ and $g^3$. 
\begin{table}
\begin{center}
\begin{tabular}{|c||c|c|c|c|c|c|c|c|c|}\hline
$g$ & $I$ & $-I$ & $\alpha$ & $\beta$ & $\beta'$ & $\gamma$ & $\gamma'$ & $-\gamma$ & $-\gamma'$\\\hline
$g^2$ & $I$ & $I$ & $-I$ & $\beta$ & $\beta$ & $\gamma'$ & $\gamma$ & $\gamma'$ & $\gamma$\\ \hline
$g^3$ & $I$ & $-I$ & $\alpha$ & $I$ & $-I$ & $\gamma'$ & $\gamma$ & $-\gamma'$ & $-\gamma$\\ \hline
\end{tabular}
\end{center}
\caption{The conjugacy classes of $g^2$ and $g^3$ in $\mathrm{SL}_2(\F_5)$. See also Appendix~\ref{s:elements}.}\label{tab:conj}
\end{table}

We need also to consider terms for the elements $\tau\cdot(g,h)$ given by the following formulas.
\begin{equation}\label{eq:term_tau}
 \begin{split}
  &\displaystyle\frac{1}{6|\pi|^2}\sum_{g,h\in\pi}(\chi_W(\tau\cdot(g,h))^3-3 \chi_W((\tau\cdot (g,h))^2) \chi_W(\tau\cdot(g,h))+2 \chi_W((\tau\cdot(g,h))^3),\\
  &\displaystyle\frac{1}{6|\pi|^2}\sum_{g,h\in\pi}(\chi_W(\tau\cdot(g,h))^3+3 \chi_W((\tau\cdot (g,h))^2) \chi_W(\tau\cdot(g,h))+2 \chi_W((\tau\cdot(g,h))^3).
\end{split}
\end{equation}
We simplify the computation as follows.
\[ \begin{split}
&\sum_{g,h}\chi_W(\tau\cdot(g,h))^3=\sum_{g,h}\Bigl(\sum_i\chi_{A_i\boxtimes A_i}(\tau\cdot(g,h))\Bigr)^3
  =\sum_{g,h}\Bigl(\sum_i\chi_{A_i}(hg)\Bigr)^3\\
&=\sum_h\sum_g\Bigl(\sum_i\chi_{A_i}(hg)\Bigr)^3=\sum_h\sum_g\Bigl(\sum_i\chi_{A_i}(g)\Bigr)^3=|\pi|\sum_g\Bigl(\sum_i\chi_{A_i}(g)\Bigr)^3,\\
&\sum_{g,h}\chi_W((\tau\cdot(g,h))^2)\chi_W(\tau\cdot(g,h))
=\sum_{g,h}\Bigl(\sum_i\chi_{A_i\boxtimes A_i}(hg,gh)\Bigr)\Bigl(\sum_j\chi_{A_j}(hg)\Bigr)\\
&=\sum_{g,h}\Bigl(\sum_i\chi_{A_i}(hg)\chi_{A_i}(gh)\Bigr)\Bigl(\sum_j\chi_{A_j}(hg)\Bigr)\\
&=\sum_{g,h}\Bigl(\sum_i\chi_{A_i}(hg)^2\Bigr)\Bigl(\sum_j\chi_{A_j}(hg)\Bigr)
=\sum_h\sum_g\Bigl(\sum_i\chi_{A_i}(g)^2\Bigr)\Bigl(\sum_j\chi_{A_j}(g)\Bigr)\\
&=|\pi|\sum_g\Bigl(\sum_i\chi_{A_i}(g)^2\Bigr)\Bigl(\sum_j\chi_{A_j}(g)\Bigr),\\
&\sum_{g,h}\chi_W((\tau\cdot(g,h))^3)=\sum_{g,h}\sum_i\chi_{A_i\boxtimes A_i}((\tau\cdot(g,h))^3)=\sum_{g,h}\sum_i\chi_{A_i}(hghghg)\\
&=|\pi|\sum_g\sum_i\chi_{A_i}(g^3).
\end{split} \]
Here, we have the identity $\chi_{A_i\boxtimes A_i}(\tau\cdot(g,h))=\chi_{A_i}(hg)$ since $\tau$ acts on $A_i\boxtimes A_i$ by the flip $x\boxtimes y\mapsto y\boxtimes x$. 
Hence (\ref{eq:term_tau}) can be computed respectively by the following formulas.
\[ \begin{split}
&\frac{1}{6|\pi|}\sum_g\Bigl\{\Bigl(
\sum_i\chi_{A_i}(g)\Bigr)^3-3\Bigl(\sum_i\chi_{A_i}(g)^2\Bigr)\Bigl(\sum_j\chi_{A_j}(g)\Bigr)+2\sum_i\chi_{A_i}(g^3)\Bigr\},\\
&\frac{1}{6|\pi|}\sum_g\Bigl\{\Bigl(
\sum_i\chi_{A_i}(g)\Bigr)^3+3\Bigl(\sum_i\chi_{A_i}(g)^2\Bigr)\Bigl(\sum_j\chi_{A_j}(g)\Bigr)+2\sum_i\chi_{A_i}(g^3)\Bigr\}.\\
\end{split} \]
Substituting the values of the characters in Table~\ref{tab:ch} into these formulas, we get the values
$21$ and $59$, respectively. By taking the averages $(33+21)/2=27$, $(71+59)/2=65$, we get the result.
\end{proof}

For $W=\mathrm{Ker}\,\ve$, the following proposition holds.
\begin{Prop}\label{prop:A_CG_Ker_e}
For any finite group $G$, we have the following.
\begin{enumerate}
\item $\calA_\Theta^\even(\C[G])=\calA_\Theta^\even(\mathrm{Ker}\,\ve)$.
\item $\calA_\Theta^\odd(\C[G])\cong\calA_\Theta^\odd(\mathrm{Ker}\,\ve)\oplus (\C\hat{G})_{\Z_2}$ as vector spaces over $\C$, where $\hat{G}$ is the set of conjugacy classes in $G$, and the $\Z_2$-action on it is the one induced by the inversion.
\end{enumerate}
\end{Prop}
\begin{proof}
Instead of the $(G\times G)\rtimes \Z_2$-invariant part, we consider the submodule of coinvariants, and apply the formula $V_H\cong (V_K)_{H/K}$ (e.g., \cite[Ch.II-2 (Exercises 3)]{Br}) for a subgroup $K$ of $H$ and an $H$-module $V$ to $K=G\times G$, $H=(G\times G)\rtimes \Z_2$. 

For (1), let $U$ be a trivial 1-dimensional $G\times G$-module and let $W$ be a $G\times G$-module. The formula $\tbigwedge^n(U\oplus W)=\bigoplus_{p+q=n}\tbigwedge^p U\otimes \tbigwedge^q W$ and $\tbigwedge^2 U=0$ gives
\[ \begin{split}
  \tbigwedge^3(U\oplus W)&=\tbigwedge^3U\oplus \tbigwedge^3W\oplus ((\tbigwedge^2 U)\otimes W)\oplus (U\otimes\tbigwedge^2W)\\
  &\cong\tbigwedge^3W\oplus \tbigwedge^2W.
\end{split} \]
If $W=\mathrm{Ker}\,\ve$ and if $\Z_2$ acts trivially on $U$, then we have $U\oplus W\cong\C[G]$ as both $G\times G$-modules and $\Z_2$-modules, and $((\tbigwedge^2 W)_{G\times G})^{\Z_2}=0$. Indeed, we have $((\tbigwedge^2 \C[G])_{G\times G})^{\Z_2}=0$, since in $(\tbigwedge^2 \C[G])_{G\times G}$, we have $[g\wedge h]=[gh^{-1}\wedge 1]=-[1\wedge gh^{-1}]=-[g^{-1}\wedge h^{-1}]$ ($g,h\in G$) and the $\Z_2$-invariant is generated by $[g\wedge h]+[g^{-1}\wedge h^{-1}]=0$. Hence we have
\[ ((\tbigwedge^3(U\oplus \mathrm{Ker}\,\ve))_{G\times G})^{\Z_2}\cong ((\tbigwedge^3\mathrm{Ker}\,\ve)_{G\times G})^{\Z_2}. \]

For (2), we use the formula $\Sym^n(U\oplus W)=\bigoplus_{p+q=n}\Sym^p U\otimes \Sym^q W$ to obtain
\[ \begin{split}
  \Sym^3(U\oplus W)&=\Sym^3U\oplus \Sym^3W\oplus ((\Sym^2 U)\otimes W)\oplus (U\otimes\Sym^2W)\\
  &\cong U\oplus \Sym^3W\oplus W\oplus \Sym^2W\cong \Sym^3W\oplus\Sym^2(U\oplus W).
\end{split} \]
Considering the case when $W=\mathrm{Ker}\,\ve$, it suffices to prove $((\Sym^2\C[G])_{G\times G})^{\Z_2}\cong (\C\hat{G})_{\Z_2}$. Let $\rho\colon\Sym^2\C[G]\to (\C\hat{G})_{\Z_2}$ be the $\C$-linear map defined by $\rho(g\cdot h)=[gh^{-1}]$ ($g,h\in G$), which is well-defined since $\rho(h\cdot g)=[hg^{-1}]=[gh^{-1}]=\rho(g\cdot h)$. 
One can check that this map induces a well-defined $\C$-linear isomorphism 
\[ \bar\rho\colon((\Sym^2\C[G])_{G\times G})^{\Z_2}\to (\C\hat{G})_{\Z_2} \]
with the inverse given by $\bar{\rho}^{-1}([x])=\frac{1}{2}([1\cdot x^{-1}]+[1\cdot x])$, which is well-defined. Indeed, $\bar{\rho}(gx\cdot hx)=[gxx^{-1}h^{-1}]=\bar{\rho}(g\cdot h)$, $\bar{\rho}(xg\cdot xh)=[xgh^{-1}x^{-1}]=[gh^{-1}]=\bar{\rho}(g\cdot h)$, $\bar{\rho}(g\cdot h)=\bar{\rho}(gh^{-1}\cdot 1)=\bar{\rho}(1\cdot gh^{-1})=\bar{\rho}(g^{-1}\cdot h^{-1})$ etc.
\end{proof}

\begin{Prop}\label{prop:ex1}
When $\pi=\mathrm{SL}_2(\F_5)$, we have the following.
\begin{enumerate}
\item $\dim \calA_\Theta^\even(\mathrm{Ker}\,\ve)=\dim \calA_\Theta^\even(\C[\pi])=27$.
\item $\dim \calA_\Theta^\odd(\mathrm{Ker}\,\ve)=56$, $\dim \calA_\Theta^\odd(\C[\pi])=65$.
\end{enumerate}
\end{Prop}
\begin{proof}
This follows from Lemma~\ref{lem:ohta} and Proposition~\ref{prop:A_CG_Ker_e}.
Note that by Lemma~\ref{lem:involution-inv} the action of $\Z_2$ on $\C\hat{\pi}$ is trivial, and hence $\dim(\C\hat{\pi})_{\Z_2}=\dim \C\hat{\pi}=9$.
\end{proof}

%%%%%%%%%%%%%%%%%%%%%%%%%%%%%%%%%%%%%%%%%%
\mysection{Example 2: Lens spaces}{s:cyclic}

If $\pi$ is the cyclic group $\Z_n=\{1,t,t^2,\ldots,t^{n-1}\}$, the exact values of the dimensions of the spaces $\calA_\Theta^{\even/\odd}(\C[\pi])$ and $\calA_\Theta^{\even/\odd}(\mathrm{Ker}\,\ve)$ can be determined with the help of the $\C$-linear maps (``weight system'')
\[ \begin{split}
 W^\even\colon\calA_\Theta^\even(\C[\pi])\to \tbigwedge^3 \C[\pi],\qquad
 W^\odd\colon\calA_\Theta^\odd(\C[\pi])\to \Sym^3\,\C[\pi]
\end{split} \]
defined respectively by
\[ \begin{split}
  &W^{\even}(t^a\wedge t^b\wedge t^c)=t^{b-a}\wedge t^{c-b}\wedge t^{a-c} + t^{a-b}\wedge t^{b-c}\wedge t^{c-a},\\
  &W^{\odd}(t^a\cdot t^b\cdot t^c)=t^{b-a}\cdot t^{c-b}\cdot t^{a-c} + t^{a-b}\cdot t^{b-c}\cdot t^{c-a},
\end{split} \]
instead of calculating the characters.
\begin{Lem}\label{lem:w_well_defined}
For $\pi=\Z_n$, the maps $W^{\even/\odd}$ are well-defined.
\end{Lem}
\begin{proof}
We need only to check that $W^{\even/\odd}$ is alternating/symmetric and is invariant under both the actions of $\pi\times \pi$ and the involution $(g,h,k)\mapsto (g^{-1},h^{-1},k^{-1})$. For $W^\even$, this can be checked as follows ($t^k\in\pi$):
\[ \begin{split}
  W^\even(t^b\wedge t^a\wedge t^c)&=t^{a-b}\wedge t^{c-a}\wedge t^{b-c}+t^{b-a}\wedge t^{a-c}\wedge t^{c-b}\\
  &=-t^{a-b}\wedge t^{b-c}\wedge t^{c-a}-t^{b-a}\wedge t^{c-b}\wedge t^{a-c}\\
  &=-W^\even(t^a\wedge t^b\wedge t^c)\quad \mbox{ etc. }\\
  W^\even(t^kt^a\wedge t^kt^b\wedge t^kt^c)&=t^{b-a}\wedge t^{c-b}\wedge t^{a-c} + t^{a-b}\wedge t^{b-c}\wedge t^{c-a}\\
  &=W^\even(t^a\wedge t^b\wedge t^c),\\
  W^\even(t^{-a}\wedge t^{-b}\wedge t^{-c})&=t^{-b+a}\wedge t^{-c+b}\wedge t^{-a+c}+t^{-a+b}\wedge t^{-b+c}\wedge t^{-c+a}\\
  &=W^\even(t^a\wedge t^b\wedge t^c).\\
\end{split} \]
The proof for $W^\odd$ is similar.
\end{proof}

Then $W^{\even/\odd}$ is an embedding into a subspace isomorphic to the space spanned by $t^p\wedge t^q\wedge t^r$ or $t^p\cdot t^q\cdot t^r$ $(0\leq p,q,r<n,\,p+q+r=0\mbox{ (mod $n$)}$) quotiented by the relation $t^p\wedge t^q\wedge t^r\sim t^{-p}\wedge t^{-q}\wedge t^{-r}$ or $t^p\cdot t^q\cdot t^r\sim t^{-p}\cdot t^{-q}\cdot t^{-r}$. 
\begin{Prop}\label{prop:ex2}
Let $\pi=\Z_n$ ($n\geq 1$), and for an integer $m\geq 0$, let $p_3(m)$ be the number of partitions of $m$ into at most three parts, namely, the number of integer solutions of the equation $x+y+z=m$ ($0\leq x\leq y\leq z$). We set $p_3(m)=0$ for $m<0$. Then we have the following.
\begin{enumerate}
\item $\dim\calA_\Theta^{\odd}(\C[\pi])=p_3(n)$.
\item $\dim\calA_\Theta^{\even}(\C[\pi])=p_3(n-6)$.
\item $\dim\calA_\Theta^{\odd}(\mathrm{Ker}\,\ve)=p_3(n-3)$.
\item $\dim\calA_\Theta^{\even}(\mathrm{Ker}\,\ve)=\dim\calA_\Theta^{\even}(\C[\pi])=p_3(n-6)$.
\end{enumerate}
\end{Prop}
\begin{proof}
For (1), we consider the unit circle $S^1=\{z\in\C\mid |z|=1\}$ and the $n$ points $1,\omega,\omega^2,\ldots,\omega^{n-1}$ on it, where $\omega=e^{2\pi\sqrt{-1}/n}$. We represent the element $t^a\cdot t^b\cdot t^c$ ($0\leq a\leq b\leq c\leq n-1$) by the three points $\omega^a,\omega^b,\omega^c$ on $S^1$, which splits $S^1$ into three (possibly degenerate) arcs of lengths $\frac{2\pi}{n}(b-a), \frac{2\pi}{n}(c-b), \frac{2\pi}{n}(a-c)$ (mod $2\pi$), respectively. Similarly, $\omega^{-a}, \omega^{-b}, \omega^{-c}$ splits $S^1$ into three arcs of lengths $\frac{2\pi}{n}(a-b), \frac{2\pi}{n}(b-c), \frac{2\pi}{n}(c-a)$ (mod $2\pi$), which are reflections of the previous triple with respect to the real axis. In this way, the subspace spanned by the values of $W^{\odd}(t^a\cdot t^b\cdot t^c)$ bijectively correspond to the space spanned by partitions of $S^1$ by three roots of 1 up to $\frac{2\pi k}{n}$-rotation and reflection. The number of such classes of partitions is exactly the number $p_3(n)$.

Proof of (2) is similar. In this case partitions are slightly restricted. First, a partition should not have a zero part since such a partition comes from $t^a\wedge t^b\wedge t^c$ such that at least two of $a,b,c$ agree. Also, a partition should have different sizes since we consider the value of the weight system in the alternating product $\tbigwedge^3\C[\pi]$. It follows that $\dim\mathrm{Im}\,W^\even$ agrees with the number of partitions of $n$ into three nonzero parts with different sizes. Such partitions $n=x+y+z$ ($0<x<y<z$) correspond bijectively to the partitions $n-(1+2+3)=(x-1)+(y-2)+(z-3)$ ($0\leq x-1\leq y-2\leq z-3$). This completes the proof.

For (3), it follows from Proposition~\ref{prop:A_CG_Ker_e} that $\dim\calA_\Theta^{\odd}(\mathrm{Ker}\,\ve)=\dim\calA_\Theta^{\odd}(\C[\pi])-\dim(\C\hat{\pi})_{\Z_2}$, where $\dim(\C\hat{\pi})_{\Z_2}$ is the number of partitions of $n$ into at most two parts. Hence $\dim\calA_\Theta^{\odd}(\mathrm{Ker}\,\ve)$ is the number of partitions of $n$ into three parts with positive sizes. Such partitions $n=x+y+z$ ($0<x\leq y\leq z$) correspond bijectively to partitions $n-3=(x-1)+(y-1)+(z-1)$ ($0\leq x-1\leq y-1\leq z-1$).

(4) follows immediately from (2) and Proposition~\ref{prop:A_CG_Ker_e}.
\end{proof}

%%%%%%%%%%%%%%%%%%%%%%%%%%%%%%%%%%%%%%%%%%%%%%
\begin{appendix}
\mysection{Elements of $\mathrm{SL}_2(\F_5)$}{s:elements}

The following is a list of all the 120 elements in $\mathrm{SL}_2(\F_5)$.
\[\scriptsize\begin{split}
&\underline{c_1=[I]}:\,\,g_{21}=\left(\begin{array}{cc}
1 & 0\\
0 & 1
\end{array}\right)\,\,\qquad \underline{c_2=[-I]}:\,\,
g_{96}=\left(\begin{array}{cc}
4 & 0\\
0 & 4
\end{array}\right)\,\,\\
\end{split}\]
\[\scriptsize\begin{split}
&\underline{c_3=[\alpha]}:\\
&g_{1}=\left(\begin{array}{cc}
0 & 1\\
4 & 0
\end{array}\right),\,\,
g_{6}=\left(\begin{array}{cc}
0 & 2\\
2 & 0
\end{array}\right),\,\,
g_{11}=\left(\begin{array}{cc}
0 & 3\\
3 & 0
\end{array}\right),\,\,
g_{16}=\left(\begin{array}{cc}
0 & 4\\
1 & 0
\end{array}\right),\,\,
g_{29}=\left(\begin{array}{cc}
1 & 1\\
3 & 4
\end{array}\right),\,\,\\
&
g_{35}=\left(\begin{array}{cc}
1 & 2\\
4 & 4
\end{array}\right),\,\,
g_{37}=\left(\begin{array}{cc}
1 & 3\\
1 & 4
\end{array}\right),\,\,
g_{43}=\left(\begin{array}{cc}
1 & 4\\
2 & 4
\end{array}\right),\,\,
g_{46}=\left(\begin{array}{cc}
2 & 0\\
0 & 3
\end{array}\right),\,\,
g_{47}=\left(\begin{array}{cc}
2 & 0\\
1 & 3
\end{array}\right),\,\,\\
&g_{48}=\left(\begin{array}{cc}
2 & 0\\
2 & 3
\end{array}\right),\,\,
g_{49}=\left(\begin{array}{cc}
2 & 0\\
3 & 3
\end{array}\right),\,\,
g_{50}=\left(\begin{array}{cc}
2 & 0\\
4 & 3
\end{array}\right),\,\,
g_{51}=\left(\begin{array}{cc}
2 & 1\\
0 & 3
\end{array}\right),\,\,
g_{56}=\left(\begin{array}{cc}
2 & 2\\
0 & 3
\end{array}\right),\,\,
\\
&g_{61}=\left(\begin{array}{cc}
2 & 3\\
0 & 3
\end{array}\right),\,\,
g_{66}=\left(\begin{array}{cc}
2 & 4\\
0 & 3
\end{array}\right),\,\,
g_{71}=\left(\begin{array}{cc}
3 & 0\\
0 & 2
\end{array}\right),\,\,
g_{72}=\left(\begin{array}{cc}
3 & 0\\
1 & 2
\end{array}\right),\,\,
g_{73}=\left(\begin{array}{cc}
3 & 0\\
2 & 2
\end{array}\right),\,\,
\\
&g_{74}=\left(\begin{array}{cc}
3 & 0\\
3 & 2
\end{array}\right),\,\,
g_{75}=\left(\begin{array}{cc}
3 & 0\\
4 & 2
\end{array}\right),\,\,
g_{76}=\left(\begin{array}{cc}
3 & 1\\
0 & 2
\end{array}\right),\,\,
g_{81}=\left(\begin{array}{cc}
3 & 2\\
0 & 2
\end{array}\right),\,\,
g_{86}=\left(\begin{array}{cc}
3 & 3\\
0 & 2
\end{array}\right),\,\,
\\
&g_{91}=\left(\begin{array}{cc}
3 & 4\\
0 & 2
\end{array}\right),\,\,
g_{104}=\left(\begin{array}{cc}
4 & 1\\
3 & 1
\end{array}\right),\,\,
g_{110}=\left(\begin{array}{cc}
4 & 2\\
4 & 1
\end{array}\right),\,\,
g_{112}=\left(\begin{array}{cc}
4 & 3\\
1 & 1
\end{array}\right),\,\,
g_{118}=\left(\begin{array}{cc}
4 & 4\\
2 & 1
\end{array}\right)\,\,
\end{split}\]
\[\scriptsize\begin{split}
&\underline{c_4=[\beta]}:\\
&g_{5}=\left(\begin{array}{cc}
0 & 1\\
4 & 4
\end{array}\right),\,\,
g_{10}=\left(\begin{array}{cc}
0 & 2\\
2 & 4
\end{array}\right),\,\,
g_{15}=\left(\begin{array}{cc}
0 & 3\\
3 & 4
\end{array}\right),\,\,
g_{20}=\left(\begin{array}{cc}
0 & 4\\
1 & 4
\end{array}\right),\,\,
g_{28}=\left(\begin{array}{cc}
1 & 1\\
2 & 3
\end{array}\right),\,\,\\
&g_{32}=\left(\begin{array}{cc}
1 & 2\\
1 & 3
\end{array}\right),\,\,
g_{40}=\left(\begin{array}{cc}
1 & 3\\
4 & 3
\end{array}\right),\,\,
g_{44}=\left(\begin{array}{cc}
1 & 4\\
3 & 3
\end{array}\right),\,\,
g_{54}=\left(\begin{array}{cc}
2 & 1\\
3 & 2
\end{array}\right),\,\,
g_{60}=\left(\begin{array}{cc}
2 & 2\\
4 & 2
\end{array}\right),\,\,\\
&g_{62}=\left(\begin{array}{cc}
2 & 3\\
1 & 2
\end{array}\right),\,\,
g_{68}=\left(\begin{array}{cc}
2 & 4\\
2 & 2
\end{array}\right),\,\,
g_{78}=\left(\begin{array}{cc}
3 & 1\\
2 & 1
\end{array}\right),\,\,
g_{82}=\left(\begin{array}{cc}
3 & 2\\
1 & 1
\end{array}\right),\,\,
g_{90}=\left(\begin{array}{cc}
3 & 3\\
4 & 1
\end{array}\right),\,\,\\
&g_{94}=\left(\begin{array}{cc}
3 & 4\\
3 & 1
\end{array}\right),\,\,
g_{105}=\left(\begin{array}{cc}
4 & 1\\
4 & 0
\end{array}\right),\,\,
g_{108}=\left(\begin{array}{cc}
4 & 2\\
2 & 0
\end{array}\right),\,\,
g_{114}=\left(\begin{array}{cc}
4 & 3\\
3 & 0
\end{array}\right),\,\,
g_{117}=\left(\begin{array}{cc}
4 & 4\\
1 & 0
\end{array}\right)\,\,
\end{split}\]
\[\scriptsize\begin{split}
&\underline{c_5=[\beta']}:\\
&g_{2}=\left(\begin{array}{cc}
0 & 1\\
4 & 1
\end{array}\right),\,\,
g_{7}=\left(\begin{array}{cc}
0 & 2\\
2 & 1
\end{array}\right),\,\,
g_{12}=\left(\begin{array}{cc}
0 & 3\\
3 & 1
\end{array}\right),\,\,
g_{17}=\left(\begin{array}{cc}
0 & 4\\
1 & 1
\end{array}\right),\,\,
g_{30}=\left(\begin{array}{cc}
1 & 1\\
4 & 0
\end{array}\right),\,\,\\
&g_{33}=\left(\begin{array}{cc}
1 & 2\\
2 & 0
\end{array}\right),\,\,
g_{39}=\left(\begin{array}{cc}
1 & 3\\
3 & 0
\end{array}\right),\,\,
g_{42}=\left(\begin{array}{cc}
1 & 4\\
1 & 0
\end{array}\right),\,\,
g_{53}=\left(\begin{array}{cc}
2 & 1\\
2 & 4
\end{array}\right),\,\,
g_{57}=\left(\begin{array}{cc}
2 & 2\\
1 & 4
\end{array}\right),\,\,\\
&g_{65}=\left(\begin{array}{cc}
2 & 3\\
4 & 4
\end{array}\right),\,\,
g_{69}=\left(\begin{array}{cc}
2 & 4\\
3 & 4
\end{array}\right),\,\,
g_{79}=\left(\begin{array}{cc}
3 & 1\\
3 & 3
\end{array}\right),\,\,
g_{85}=\left(\begin{array}{cc}
3 & 2\\
4 & 3
\end{array}\right),\,\,
g_{87}=\left(\begin{array}{cc}
3 & 3\\
1 & 3
\end{array}\right),\,\,\\
&g_{93}=\left(\begin{array}{cc}
3 & 4\\
2 & 3
\end{array}\right),\,\,
g_{103}=\left(\begin{array}{cc}
4 & 1\\
2 & 2
\end{array}\right),\,\,
g_{107}=\left(\begin{array}{cc}
4 & 2\\
1 & 2
\end{array}\right),\,\,
g_{115}=\left(\begin{array}{cc}
4 & 3\\
4 & 2
\end{array}\right),\,\,
g_{119}=\left(\begin{array}{cc}
4 & 4\\
3 & 2
\end{array}\right)\,\,
\end{split}\]
\[\scriptsize\begin{split}
&\underline{c_6=[\gamma]}:\\
&g_{3}=\left(\begin{array}{cc}
0 & 1\\
4 & 2
\end{array}\right),\,\,
g_{18}=\left(\begin{array}{cc}
0 & 4\\
1 & 2
\end{array}\right),\,\,
g_{22}=\left(\begin{array}{cc}
1 & 0\\
1 & 1
\end{array}\right),\,\,
g_{25}=\left(\begin{array}{cc}
1 & 0\\
4 & 1
\end{array}\right),\,\,
g_{26}=\left(\begin{array}{cc}
1 & 1\\
0 & 1
\end{array}\right),\,\,\\
&g_{41}=\left(\begin{array}{cc}
1 & 4\\
0 & 1
\end{array}\right),\,\,
g_{55}=\left(\begin{array}{cc}
2 & 1\\
4 & 0
\end{array}\right),\,\,
g_{67}=\left(\begin{array}{cc}
2 & 4\\
1 & 0
\end{array}\right),\,\,
g_{77}=\left(\begin{array}{cc}
3 & 1\\
1 & 4
\end{array}\right),\,\,
g_{95}=\left(\begin{array}{cc}
3 & 4\\
4 & 4
\end{array}\right),\,\,\\
&g_{102}=\left(\begin{array}{cc}
4 & 1\\
1 & 3
\end{array}\right),\,\,
g_{120}=\left(\begin{array}{cc}
4 & 4\\
4 & 3
\end{array}\right)\,\,
\end{split} \]
\[\scriptsize\begin{split}
&\underline{c_7=[\gamma']}:\\
&g_{8}=\left(\begin{array}{cc}
0 & 2\\
2 & 2
\end{array}\right),\,\,
g_{13}=\left(\begin{array}{cc}
0 & 3\\
3 & 2
\end{array}\right),\,\,
g_{23}=\left(\begin{array}{cc}
1 & 0\\
2 & 1
\end{array}\right),\,\,
g_{24}=\left(\begin{array}{cc}
1 & 0\\
3 & 1
\end{array}\right),\,\,
g_{31}=\left(\begin{array}{cc}
1 & 2\\
0 & 1
\end{array}\right),\,\,\\
&g_{36}=\left(\begin{array}{cc}
1 & 3\\
0 & 1
\end{array}\right),\,\,
g_{58}=\left(\begin{array}{cc}
2 & 2\\
2 & 0
\end{array}\right),\,\,
g_{64}=\left(\begin{array}{cc}
2 & 3\\
3 & 0
\end{array}\right),\,\,
g_{84}=\left(\begin{array}{cc}
3 & 2\\
3 & 4
\end{array}\right),\,\,
g_{88}=\left(\begin{array}{cc}
3 & 3\\
2 & 4
\end{array}\right),\,\,\\
&g_{109}=\left(\begin{array}{cc}
4 & 2\\
3 & 3
\end{array}\right),\,\,
g_{113}=\left(\begin{array}{cc}
4 & 3\\
2 & 3
\end{array}\right)\,\,
\end{split} \]
\[\scriptsize \begin{split} 
&\underline{c_8=[-\gamma]}:\\
&g_{4}=\left(\begin{array}{cc}
0 & 1\\
4 & 3
\end{array}\right),\,\,
g_{19}=\left(\begin{array}{cc}
0 & 4\\
1 & 3
\end{array}\right),\,\,
g_{27}=\left(\begin{array}{cc}
1 & 1\\
1 & 2
\end{array}\right),\,\,
g_{45}=\left(\begin{array}{cc}
1 & 4\\
4 & 2
\end{array}\right),\,\,
g_{52}=\left(\begin{array}{cc}
2 & 1\\
1 & 1
\end{array}\right),\,\,\\
&g_{70}=\left(\begin{array}{cc}
2 & 4\\
4 & 1
\end{array}\right),\,\,
g_{80}=\left(\begin{array}{cc}
3 & 1\\
4 & 0
\end{array}\right),\,\,
g_{92}=\left(\begin{array}{cc}
3 & 4\\
1 & 0
\end{array}\right),\,\,
g_{97}=\left(\begin{array}{cc}
4 & 0\\
1 & 4
\end{array}\right),\,\,
g_{100}=\left(\begin{array}{cc}
4 & 0\\
4 & 4
\end{array}\right),\,\,\\
&g_{101}=\left(\begin{array}{cc}
4 & 1\\
0 & 4
\end{array}\right),\,\,
g_{116}=\left(\begin{array}{cc}
4 & 4\\
0 & 4
\end{array}\right)\,\,\\
&\underline{c_9=[-\gamma']}:\\
&g_{9}=\left(\begin{array}{cc}
0 & 2\\
2 & 3
\end{array}\right),\,\,
g_{14}=\left(\begin{array}{cc}
0 & 3\\
3 & 3
\end{array}\right),\,\,
g_{34}=\left(\begin{array}{cc}
1 & 2\\
3 & 2
\end{array}\right),\,\,
g_{38}=\left(\begin{array}{cc}
1 & 3\\
2 & 2
\end{array}\right),\,\,
g_{59}=\left(\begin{array}{cc}
2 & 2\\
3 & 1
\end{array}\right),\,\,\\
&g_{63}=\left(\begin{array}{cc}
2 & 3\\
2 & 1
\end{array}\right),\,\,
g_{83}=\left(\begin{array}{cc}
3 & 2\\
2 & 0
\end{array}\right),\,\,
g_{89}=\left(\begin{array}{cc}
3 & 3\\
3 & 0
\end{array}\right),\,\,
g_{98}=\left(\begin{array}{cc}
4 & 0\\
2 & 4
\end{array}\right),\,\,
g_{99}=\left(\begin{array}{cc}
4 & 0\\
3 & 4
\end{array}\right),\,\,\\
&g_{106}=\left(\begin{array}{cc}
4 & 2\\
0 & 4
\end{array}\right),\,\,
g_{111}=\left(\begin{array}{cc}
4 & 3\\
0 & 4
\end{array}\right)\,\,
\end{split} \]
\par\bigskip

\end{appendix}

%%%%%%%%%%%%%%%%%%%%%%%%%%%%%%%

\end{document}